\documentclass[12pt]{amsart}

\usepackage{amssymb}
\usepackage{latexsym}
\usepackage{exscale}
\usepackage{lscape}

\usepackage[colorlinks=true, pdfstartview=FitV, linkcolor=blue,
citecolor=blue, urlcolor=blue]{hyperref}

\newtheorem{theorem}{Theorem}
\newtheorem{lemma}[theorem]{Lemma}

\newtheorem{corol}[theorem]{Corollary}

\theoremstyle{definition}

\theoremstyle{remark}
\newtheorem{remark}[theorem]{Remark}

\renewcommand{\emptyset}{\mbox{\textup{\O}}}

\newcommand{\re}{\mathbb R}

\DeclareMathOperator{\supp}{supp}

\newcommand{\subRn}{{{\mathbb R}^n}}

\newcommand{\la}{\langle}

\newcommand{\Q}{\mathcal{Q}}

\def\Xint#1{\mathchoice
  {\XXint\displaystyle\textstyle{#1}}%
  {\XXint\textstyle\scriptstyle{#1}}%
  {\XXint\scriptstyle\scriptscriptstyle{#1}}%
  {\XXint\scriptscriptstyle\scriptscriptstyle{#1}}%
  \!\int}
\def\XXint#1#2#3{{\setbox0=\hbox{$#1{#2#3}{\int}$}
    \vcenter{\hbox{$#2#3$}}\kern-.5\wd0}}

\def\avgint{\Xint-}

\allowdisplaybreaks

\begin{document}

\title[Sharp weighted estimates]{Sharp weighted estimates for approximating dyadic operators}

\author{David Cruz-Uribe, SFO}
\address{David Cruz-Uribe, SFO\\
Dept. of Mathematics \\ Trinity College \\
Hartford, CT 06106-3100, USA} \email{david.cruzuribe@trincoll.edu}

\author{Jos\'e Mar{\'\i}a Martell}

\address{Jos\'e Mar{\'\i}a Martell
\\
Instituto de Ciencias Matem\'aticas CSIC-UAM-UC3M-UCM
\\
Consejo Superior de Investigaciones Cient{\'\i}ficas
\\
C/ Serrano 121
\\
E-28006 Madrid, Spain} \email{chema.martell@uam.es}

\author{Carlos P\'erez}
\address{Carlos P\'erez
\\
Departamento de An\'alisis Matem\'atico, Facultad de Matem\'aticas\\
Universidad de Sevilla, 41080 Sevilla, Spain}
\email{carlosperez@us.es}

\subjclass{42B20, 42B25} \keywords{$A_p$ weights, Haar shift
operators singular  integral operators, Hilbert transform, Riesz
transforms, Beurling-Ahlfors operator, dyadic square function,
vector-valued maximal operator}

\thanks{The first author was supported by a grant from the Faculty
  Research Committee and the Stewart-Dorwart Faculty Development Fund
  at Trinity College; the first and third authors are supported by
  grant MTM2009-08934 from the Spanish Ministry of Science and
  Innovation; the second author is supported by grant MTM2007-60952
  from the same institution and by CSIC PIE
  200850I015. }

\begin{abstract}
We give a new proof of the sharp weighted $L^p$ inequality
\[ \|T\|_{L^p(w)} \leq
C_{n,T}\,[w]_{A_p}^{\max\left(1,\frac{1}{p-1}\right)},
\]
where $T$ is the Hilbert transform, a Riesz transform, the
Beurling-Ahlfors operator  or any operator that can be approximated
by Haar shift operators. Our proof avoids the Bellman function
technique and two weight norm inequalities. We use instead a recent
result due to A.~Lerner~\cite{lernerP2009} to estimate the
oscillation of dyadic operators.

The method we use is flexible enough to obtain the sharp one-weight
result for other important operators as well as a very sharp
two-weight bump type result for $T$ as can be found in
\cite{cruz-uribe-martell-perez2010}.

\end{abstract}

\date{January 25, 2010}

\maketitle

\section{Introduction}

One weight norm inequalities for singular integrals of the form
\[ \|Tf\|_{L^p(w)} \leq C\|f\|_{L^p(w)} \qquad w \in A_p, \]
have a long history, beginning with the work of Hunt, Muckenhoupt
and Wheeden \cite{MR0312139} for the Hilbert transform. (See
Duoandikoetxea~\cite{MR1800316} for a concise history.)  The
constant $C$ depends on the $A_p$ constant of the weight $w$:
\[ [w]_{A_p} =\sup_Q \avgint_Q w(x)\,dx \left(\avgint_Q
 w(x)^{1-p'}\,dx\right)^{p-1}. \]
An interesting question is the exact dependence on the $A_p$
constant. This was first investigated by Buckley~\cite{MR1124164}.
More recently, this problem has attracted renewed attention because
of the work of Astala, Iwaniec and Saksman~\cite{MR1815249}.  They
proved sharp regularity results for solutions to the Beltrami
equation, assuming that the operator norm of the Beurling-Ahlfors
transform growths linearly in terms of the $A_2$ constant. This was
proved by S. Petermichl and A. Volberg~\cite{MR1894362} and by
Petermichl~\cite{MR2354322,MR2367098} for the Hilbert transform and
the Riesz Transforms. In these papers it has been shown that if $T$
is any of these operators, then
\begin{equation}\label{pet}
\|T\|_{L^p(w)}\le
c_{p,n}\,[w]_{A_p}^{\max\left(1,\frac{1}{p-1}\right)}
\qquad1<p<\infty,
\end{equation}
and the exponent $\max\left(1,\frac{1}{p-1}\right)$ is best
possible. It has been conjectured that the same estimate
holds for any Calder\'on-Zygmund operator $T$. By the sharp version
of the Rubio de Francia extrapolation theorem due to
Dragi{\v{c}}evi{\'c} {\em et al.}~\cite{MR2140200}, it suffices to
prove this inequality for $p=2$, namely
\begin{equation}\label{CZA2conjecture}
\|T\|_{L^2(w)}\le c_{n}\, [w]_{A_2}.
\end{equation}
In each of the known cases, the proof used a technique developed by
Petermichl~\cite{MR1756958} to reduce the problem to proving the
analogous inequality for a corresponding Haar shift operator.  The
norm inequalities for these dyadic operators were then proved using
Bellman function techniques.

Recently, Lacey, Petermichl and
Reguera-Rodriguez~\cite{lacey-petermichl-reguera-rodriquezP2009}
gave a proof of the sharp $A_2$ constant for a large family of Haar
shift operators that includes all of the dyadic operators needed for
the above results.  Their proof avoids the use of Bellman functions,
and instead uses a deep, two-weight ``$Tb$ theorem'' for Haar shift
operators due to Nazarov, Treil and Volberg~\cite{MR2407233}.

We give a different proof that avoids both Bellman functions and
two-weight norm inequalities such as the $Tb$ theorem.  Instead, we
use a very interesting decomposition argument using local mean
oscillation recently developed by Lerner~\cite{lernerP2009}.

An important advantage of our approach is that it also yields the
optimal sharp one weight norm inequalities for other operators such
as dyadic square functions and paraproducts, maximal singular
integrals and the vector-valued maximal function of C.
Fefferman-Stein. Moreover it also gives very sharp two weight
``$A_p$ bump'' type conditions that improve results gotten by the
authors in~\cite{MR2351373}.  All these results can be found in
\cite{cruz-uribe-martell-perez2010}. Key to our approach is that the
operators are either dyadic or can be approximated by dyadic
operators (e.g., by the Haar shift operators defined below).  Thus
all these results will extend to any operator that can be
approximated in this way.

To state our result we first give some definitions
following~\cite{lacey-petermichl-reguera-rodriquezP2009} and
consider simultaneously a family of dyadic operators---the Haar
shift operators---that contains all the operators we are interested
in.

Let $\Delta$ be the set of dyadic cubes in $\re^n$.  For our
arguments we properly need to consider the sets $\Delta_{s,t}$,
$s\in \re^n$, $t>0$, of translations and dilations of dyadic cubes.
However, it will be immediate that all of our arguments for dyadic
cubes extend to these more general families, so without loss of
generality we will restrict ourselves to dyadic cubes.

We define a Haar function on a cube $Q\in \Delta$ to be a function
$h_Q$ such that
\begin{enumerate}
\item $\supp(h_Q) \subset Q$;
\item if $Q'\in \Delta$ and $Q'\subsetneq Q$, then $h_Q$ is constant on $Q'$;
\item $\|h_Q\|_\infty \leq |Q|^{-1/2}$;
\item $\int_Q h_Q(x)\,dx = 0$.
\end{enumerate}

Given an integer $\tau\ge 0$, a Haar shift operator of index $\tau$
is an operator of the form
\[ H_\tau f(x) = \sum_{Q\in \Delta} \sum_{\substack{Q',Q''\in \Delta(Q)\\ 2^{-\tau n}|Q|\leq |Q'|,|Q''|}} a_{Q',Q''} \la f, h_{Q'} \rangle h_{Q''}(x), \]
where $a_{Q',Q''}$ is a constant such that
\[ |a_{Q',Q''}| \leq C\,\left( \frac{|Q'|}{|Q|}\frac{|Q''|}{|Q|}\right)^{1/2}. \]
We say that $H_\tau$ is a CZ Haar shift operator if it is bounded on
$L^2$.

An important example of a Haar shift operator when $n=1$ is the Haar
shift (also known as the dyadic Hilbert transform)  $H^d$, defined
by
\[ H^df(x) = \sum_{I\in \Delta} \langle f, h_I\rangle \big(
h_{I_-}(x) - h_{I_+}(x)\big), \]
where, as before, given a dyadic interval $I$, $I_+$ and $I_-$ are
its right and left halves, and
\[ h_I(x)  =  |I|^{-1/2}\big(\chi_{I_-}(x) - \chi_{I_+}(x)\big). \]
Clearly $h_I$ is a Haar function on $I$ and one can write $H^d$ as a
Haar shift operator of index $\tau=1$ with $a_{I',I''}=\pm 1$ for
$I'=I$, $I''=I_{\pm}$ and $a_{I',I''}=0$ otherwise. These are the
operators used by Petermichl~\cite{MR1756958,MR2354322} to
approximate the Hilbert transform.  More precisely, she used the
family of operators $H^d_{s,t}$, $s\in \re$, $t>0$, which are
defined as above but with the dyadic grid replaced by its
translation by $s$ and dilation by $t$.    The Hilbert transform is
then the limit of integral averages of these operators, so norm
inequalities for $H$ follow from norm inequalities for $H^d_{s,t}$
by Fatou's lemma and Minkowski's inequality.    Similar
approximations hold for the Riesz transforms and Beurling-Ahlfors
operator, and we refer the reader to \cite{MR2367098,MR1894362} for
more details.

\bigskip

We can now state our main result.

\begin{theorem} \label{thm:main-thm}
Let $H_\tau$ be a CZ Haar shift operator where $\tau \geq 0$ is an
integer. Then for every $w\in A_2$,
\[ \|H_\tau\|_{L^2(w)} \leq C(\tau, n) [w]_{A_2}. \]
As a consequence, the same norm inequality holds for the Hilbert
transform, the Riesz transforms, and the Beurling-Ahlfors operator.
\end{theorem}

If we apply the sharp version of the Rubio de Francia extrapolation
theorem~\cite{MR2140200} mentioned above, we get sharp $L^p$
estimates.

\begin{corol} \label{cor.from.main-thm}
Let $H_\tau$ as above and let $1<p<\infty$. Then for every $w\in
A_p$,
\[ \|H_\tau\|_{L^p(w)} \leq C(\tau, n,p) \,[w]_{A_p}^{\max\left(1,\frac{1}{p-1}\right)}. \]
As a consequence, the same norm inequality holds for the Hilbert
transform, the Riesz transforms, and the Beurling-Ahlfors operator.
\end{corol}

The remainder of this paper is organized as follows.  In
Section~\ref{section:lerner} we give some preliminary definitions
and state Lerner's decomposition theorem.  In
Section~\ref{section:dyadic} we prove an estimate which allows us to
apply this decomposition to the CZ Haar shift operators.  In
Section~\ref{section:main-proof} we prove
Theorem~\ref{thm:main-thm}.

\section{Local mean oscillation}
\label{section:lerner}

We begin with some basic definitions.  We follow the terminology and
notation of Lerner~\cite{lernerP2009}, which in turn is based on
Fujii~\cite{MR946637,MR1115188} and Jawerth and
Torchinsky~\cite{MR779906}. We note in passing that many of the
underlying ideas originated in the work of Carleson~\cite{MR0477058}
and Garnett and Jones~\cite{MR658065}.

Hereafter we assume that all functions $f$ are measurable and
finite-valued almost everywhere. Given a cube $Q$ and $\lambda$,
$0<\lambda<1$, define the local mean oscillation of $f$ on $Q$ by
\[ \omega_\lambda(f,Q) = \inf_{c\in \re }
\big((f-c)\chi_Q\big)^*(\lambda |Q|), \]
where $f^*$ is the non-increasing rearrangement of $f$.  The local
sharp maximal function of $f$ relative to $Q$ is then defined by
\[ M^\#_{\lambda,Q} f(x) = \sup_{\substack{Q'\ni x\\Q'\subset Q}}
\omega_\lambda(f,Q).  \]

A median value of $f$ on $Q$ is a (possibly not unique) number
$m_f(Q)$ such that both
\begin{gather*}
| \{ x\in Q : f(x) > m_f(Q) \}| \leq \frac{|Q|}{2}, \\
| \{ x\in Q : f(x) < m_f(Q) \}| \leq \frac{|Q|}{2}.
\end{gather*}
The median plays the same role for the local sharp maximal function
as the mean does for the C.~Fefferman-Stein sharp maximal function.
More precisely, for each $\lambda$, $0<\lambda \leq 1/2$,
\[ \omega_\lambda(f,Q) \leq  \big((f-m_f(Q))\chi_Q\big)^*(\lambda
|Q|) \leq 2\omega_\lambda(f,Q).\]

To estimate the median and the local mean oscillation we need the
following properties that follow from the definition of
rearrangements.  For any function $f$,  $\lambda$,  $0<\lambda<1$,
$p$, $0<p<\infty$, and cube $Q$,
\begin{gather}
(f\chi_{Q })^*(\lambda |Q|)\le \lambda^{-1/p}\,
\|f\|_{L^{p,\infty}(Q,|Q|^{-1}dx)}, \label{eqn:mean-est1} \\
(f\chi_{Q })^*(\lambda |Q|)\le
\left(\frac{1}{\lambda|Q|}\int_Q|f|^p\,dx\right)^{1/p}.
\label{eqn:mean-est2}
\end{gather}

Inequality \eqref{eqn:mean-est1} is central to our proofs as  it
allows us to use  weak $(1,1)$ inequalities directly in our
estimates.  By way of comparison, in~\cite{MR2351373} a key
technical difficulty resulted from having to use Kolmogorov's
inequality rather than the weak $(1,1)$ inequality for a singular
integral. Overcoming this is the reason the results there were
limited to log bumps.

Finally, from the definition of rearrangements we have that
\begin{equation}
|m_f(Q)| \leq (f\chi_Q)^*(|Q|/2), \label{eqn:median-f*}
\end{equation}
and so by \eqref{eqn:mean-est2}, if $f\in L^p$ for any $p>0$, then
$m_f(Q)\rightarrow 0$ as $|Q|\rightarrow \infty$.

\bigskip

Finally, to state Lerner's decomposition theorem, we use the
following notation. Given a cube $Q_0$, let $\Delta(Q_0)$ be the
collection of dyadic cubes relative to $Q_0$. And given $Q\in
\Delta(Q_0)$, $Q\neq Q_0$, let $\widehat{Q}$ be its dyadic parent:
the unique dyadic cube containing $Q$ whose side-length is twice
that of $Q$.

\begin{theorem}  \label{thm:lerner}
\textup{(\cite{lernerP2009})}\; Given a measurable function $f$ and
a cube $Q_0$, for each $k\geq 1$ there exists a (possibly empty)
collection of pairwise disjoint cubes $\{Q_j^k\} \subset
\Delta(Q_0)$ such that if $\Omega_k = \bigcup_j Q_j^k$, then
$\Omega_{k+1}\subset \Omega_k$ and $|\Omega_{k+1}\cap Q_j^k|\leq
\frac{1}{2}|Q_j^k|$.   Furthermore, for almost every $x\in Q_0$,
\[ |f(x)-m_f(Q_0)| \leq 4M^\#_{\frac{1}{4},Q_0}f(x) +
4\sum_{k,j}
\omega_{\frac{1}{2^{n+2}}}(f,\widehat{Q}_j^k)\chi_{Q_j^k}(x). \]
\end{theorem}

\begin{remark} \label{rem:disjoint}
If for all $j$ and $k$ we define $E_j^k=Q_j^k\setminus
\Omega_{k+1}$, then the sets $E_j^k$ are pairwise disjoint and
$|E_j^k|\geq \frac{1}{2}|Q_j^k|$.
\end{remark}

\begin{remark}
 Though it is not explicit in \cite{lernerP2009}, it follows at once
 from the proof that we can
 replace $M^\#_{\frac{1}{4},Q_0}$ by the corresponding dyadic
 operator $M^{\#,d}_{\frac{1}{4},Q_0}$, where
\[ M^{\#,d}_{\lambda,Q}f(x) = \sup_{x\in Q'\in \Delta(Q)} \omega_\lambda (f,Q'). \]
\end{remark}

Intuitively, one may think of the cubes $\{Q_j^k\}$ as being the
analog of the Calder\'on-Zygmund cubes for the function $f-m_f(Q_0)$
but defined with respect to the median instead of the mean.

\section{Local mean oscillation of the Haar shift operators}
\label{section:dyadic}

To apply Theorem~\ref{thm:lerner} to the Haar shift operators we
need two lemmas.  The first is simply that CZ Haar shift operators
satisfy a weak $(1,1)$ inequality.  The proof of this is known but
we could not find it in the literature and it is explicit in
\cite{cruz-uribe-martell-perez2010}. Here and below we will use the
following notation: given an integer $\tau\ge 0$ and a dyadic cube
$Q$, let $Q^\tau$ denote its $\tau$-th generation ``ancestor'': that
is, the unique dyadic cube $Q^\tau$ containing $Q$ such that
$|Q^\tau|=2^{\tau n}|Q|$.

\begin{lemma}
Given an integer $\tau\geq 0$,  there exists a constant $C_{\tau,n}$
such that for every $t>0$,
\[ \|H_\tau f \|_{L^{1,\infty}(\re^n)}  \leq C_{\tau,n}\,\int_{\re^n} |f(x)|\,dx. \]
\end{lemma}

Our second lemma is a key estimate that is sharper variant of a
result known for Calder\'on-Zygmund singular integrals (see
\cite{MR779906}) and whose proof is similar.  For completeness we
include the details.

\begin{lemma} \label{lemma:osc-est}
Given $\tau \geq 0$, let $H_\tau$ be a CZ Haar shift operator. Fix
$\lambda$,  $0<\lambda \leq 1/2$. Then for any function $f$,  every
dyadic cube $Q_0$, and every $x\in Q_0$,
\begin{gather*}
\omega_\lambda (H_\tau f,Q_0) \leq
C_{\tau,n,\lambda}\,\avgint_{Q_0^\tau}|f(x)|\,dx, \\
M^{\#,d}_{\lambda,Q_0}(H_\tau f)(x) \leq C_{\tau,n,\lambda} \,
M^df(x).
\end{gather*}
\end{lemma}

\begin{proof}
It suffices to prove the first inequality; the second follows
immediately from definition of $M^{\#,d}_{\lambda,Q_0}$. Fix $Q_0$
and write $H_\tau$ as the sum of two operators:
\[ H_\tau f(x) = H_\tau(f\chi_{Q_0^\tau})(x) + H_\tau(f\chi_{\re^n\setminus Q_0^\tau})(x). \]
We claim the second term is constant for all $x\in Q_0$.  Let $Q$ be
any dyadic cube.  Then the corresponding term in the sum defining
$H_\tau(f\chi_{\re^n\setminus Q_0^\tau})(x)$ is
\begin{equation} \label{eqn:osc-est1}
\sum_{\substack{Q',Q''\in\Delta(Q)\\ 2^{-\tau n}|Q|\leq |Q'|,|Q''|}}
a_{Q',Q''} \la f\chi_{\re^n\setminus Q_0^\tau}, h_{Q'} \rangle
h_{Q''}(x).
\end{equation}
We may assume that $Q''\cap Q_0\neq \emptyset$ (otherwise we get a
zero term); since $Q''\subset Q$, this implies that $Q\cap
Q_0^\tau\neq \emptyset$.   Similarly, we have $Q\cap (\re^n\setminus
Q_0^\tau)\neq \emptyset$.  Therefore, $Q_0^\tau \subsetneq Q$, so
$|Q_0| <2^{-\tau n}|Q|\leq |Q''|$.  Hence, $Q_0\subsetneq Q''$ and
$h_{Q''}$ is constant on $Q_0$. Thus, \eqref{eqn:osc-est1} does not
depend on $x$ and so is constant on $Q_0$.

Denote this constant by $H_\tau f(Q_0)$; then
\[
|\{ x \in Q_0 : |H_\tau f(x)- H_\tau f(Q_0)|> t \}| = |\{ x \in Q_0
: |H_\tau(f\chi_{Q_0^\tau})(x)|>t \}|.
\]
Since $H_\tau$ is a CZ Haar shift operator it is weak $(1,1)$.
Therefore, by inequality~\eqref{eqn:mean-est1},
\begin{multline*}
\omega_\lambda (H_\tau f, Q_0)   \leq \big((H_\tau f- H_\tau
f(Q_0))\chi_{Q_0}\big)^*(\lambda |Q_0|) \\
\leq
\lambda^{-1}\|H_\tau(f\chi_{Q_0^\tau})\|_{L^{1,\infty}(Q_0,|Q_0|^{-1}dx)}
\leq \frac{C_{\tau,n}}{\lambda} \avgint_{Q_0^\tau} |f(x)|\,dx.
\end{multline*}
\end{proof}

\section{The proof of Theorem~\ref{thm:main-thm}}
\label{section:main-proof}

\begin{proof}[Proof of Theorem \ref{thm:main-thm}]

Fix $w\in A_2$ and fix $f$. By a standard approximation argument we
may assume without loss of generality that $f$ is bounded and has
compact support.  Let $\re^n_j$, $1\le j\le 2^n$, denote the
$n$-dimensional quadrants in $\re^n$: that is, the sets
$I^{\pm}\times I^{\pm}\times\cdots\times I^{\pm}$ where $I^+ =
[0,\infty)$ and $I^-=(-\infty,0)$.

For each $j$, $1\le j\le 2^n$, and for each $N>0$ let $Q_{N,j}$ be
the dyadic cube adjacent to the origin of side length $2^N$ that is
contained in $\re^n_j$. Since  $Q_{N,j}\in \Delta$,
$\Delta(Q_N)\subset \Delta$. Since $H^\tau$ is a CZ shift operator
its adjoint is as well; thus, $H^\tau$ is bounded on $L^p$,
$1<p<\infty$.    In particular, by \eqref{eqn:median-f*} and
\eqref{eqn:mean-est2}, $m_{H_\tau f} (Q_{N,j})\rightarrow 0$ as
$N\rightarrow \infty$.   Therefore, by Fatou's lemma and Minkowski's
inequality,
\[ \|H_\tau f\|_{L^2(w)} \leq \liminf_{N\rightarrow \infty}
\sum_{j=1}^{2^n} \left(\int_{Q_{N,j}} |H_\tau f(x) - m_{H_\tau
f}(Q_{N,j})|^2w(x)\,dx\right)^{1/2}. \]
Hence, it will suffice to prove that each term in the sum on the
right is bounded by $C_{\tau,n}[w]_{A_2}\|f\|_{L^2(w)}$.

Fix $j$ and let $Q_N=Q_{N,j}$.  By Theorem~\ref{thm:lerner} and
Lemma~\ref{lemma:osc-est}, for every $x\in Q_N$ we have that
\begin{align}\label{eqn:Htau-decomp}
& |H_\tau f(x) - m_{H_\tau f}(Q_N)|
\\ \nonumber
&\qquad \qquad \le 4\,M^{\#,d}_{\frac{1}{4},Q_N}(H_\tau f)(x) +
4\sum_{j,k} \omega_{\frac{1}{2^{n+2}}}(H_\tau
f,\widehat{Q}_j^k)\chi_{Q_j^k}(x)
\\ \nonumber
&\qquad \qquad \le C_{\tau,n}\, M f(x) + C_{\tau,n}\,\sum_{j,k}
\left(\avgint_{P_j^k} |f(x)|\,dx\right)\, \chi_{Q_j^k}(x)
\\ \nonumber
& \qquad \qquad = C_{\tau,n}\, M f(x) + C_{\tau,n}\, F(x),
\end{align}
where $P_j^k=(\widehat{Q}_j^k)^\tau$. We get the
desired estimate for the first term from Buckley's
theorem~\cite{MR1124164} with $p=2$:
\begin{align*}
\|M f\|_{L^2(Q_N,w)} \le \|M f\|_{L^2(w)} \leq C_{n}\,[w]_{A_2}
\|f\|_{L^2(w)}.
\end{align*}

\medskip

To estimate $F$ we use duality.  Fix a non-negative function $h\in
L^2(w)$ with $\|h\|_{L^2(w)}=1$. We use the weighted dyadic maximal
operator defined by
\[ M^d_\sigma g(x) = \sup_{x\in Q\in\Delta} \frac{1}{\sigma(Q)}\int_Q |f(x)|\sigma(x)\,dx. \]
where $\sigma$ is a weight (i.e., locally integrable and positive
a.e.). In particular  we use that $M^d_\sigma$ is bounded on
$L^2(\sigma)$ with constant bounded by $2$ (see \cite[Chapter 1,
Exercise 1.3.3]{MR2445437}). Therefore, by Remark~\ref{rem:disjoint}
we have that
\begin{align*}
\int_{Q_N} F(x)\,h(x)\,w(x)\,dx & = C_{\tau,n}\, \sum_{j,k}
\avgint_{P_j^k} |f(x)|\,dx
\int_{Q_j^k} w(x)h(x)\,dx  \\
& \leq  2 \cdot 2^{(\tau+1)n}\sum_{j,k}
\frac{w(P_j^k)}{|P_j^k|}\frac{w^{-1}(P_j^k)}{|P_j^k|} \; |E_j^k|\\
& \qquad \times \frac{1}{w^{-1}(P_j^k)} \int_{P_j^k} |f(x)|w(x)
w(x)^{-1}\,dx \\
& \qquad \times \frac{1}{w(Q_j^k)}\int_{Q_j^k} h(x)w(x)\,dx \\
& \leq C_{\tau,n}\, [w]_{A_2}\sum_{j,k} \int_{E_j^k} M^d_{w^{-1}}(fw)(x) M^d_w(h)(x)\,dx \\
& \leq C_{\tau,n}\, [w]_{A_2}\int_\subRn  M^d_{w^{-1}} (fw)(x)
M^d_w(h)(x)\,dx \\
& \leq C_{\tau,n}\,[w]_{A_2}
\left(\int_\subRn M^d_{w^{-1}} (fw)(x)^2 w(x)^{-1}\,dx\right)^{1/2} \\
& \qquad \qquad \times \left(\int_\subRn M^d_{w} (h)(x)^2 w(x)\,dx\right)^{1/2} \\
& \leq C_{\tau,n}\,[w]_{A_2}
\left(\int_\subRn |f(x)w(x)|^2 w(x)^{-1}\,dx\right)^{1/2} \\
& \qquad \qquad \times \left(\int_\subRn h(x)^2 w(x)\,dx\right)^{1/2} \\
& = C_{\tau,n}\,[w]_{A_2} \left(\int_\subRn
|f(x)|^2w(x)\,dx\right)^{1/2}.
\end{align*}
If we take the supremum over all such functions $h$, we conclude
that
$$
\|F\|_{L^2(Q_N,w)} \le C_{\tau,n}\,[w]_{A_2}\,\|f\|_{L^2(w)}.
$$
Combining our estimates we have that
\[
\left(\int_{Q_N} |H_\tau f(x)-m_{H_\tau f}(Q_N)|^2
w(x)\,dx\right)^{1/2} \leq C_{\tau,n}\,[w]_{A_2} \|f\|_{L^2(w)},
\]
and this completes the proof.

\end{proof}

\bibliographystyle{plain}
\bibliography{dyadic-hilbert}

\end{document}